\documentclass[preprint,12pt]{elsarticle}

\usepackage{graphicx}
\usepackage{pifont,latexsym,ifthen,amsthm,rotating,calc,textcase,booktabs}
\usepackage{amsfonts,amssymb,amsbsy,amsmath}
\newtheorem{theorem}{Theorem}[section]
\newtheorem{lemma}[theorem]{Lemma}

\newtheorem{proposition}[theorem]{Proposition}

\usepackage{natbib}

\usepackage{todonotes}
\usepackage{float}
\usepackage{subfigure}
\usepackage{hyperref}
\hypersetup{
    colorlinks=true, 
    linktoc=all,     
    linkcolor=blue,  
    citecolor=blue
}

\begin{document}



\begin{frontmatter}
\title{The   Fokas-Lenells equation on the  line: Global well-posedness   with solitons}


\author[inst2]{Qiaoyuan Cheng}

\author[inst2]{Engui Fan$^{*,}$  }

\address[inst2]{ School of Mathematical Sciences and Key Laboratory of Mathematics for Nonlinear Science, Fudan University, Shanghai,200433, China\\
* Corresponding author and e-mail address: faneg@fudan.edu.cn  }





\begin{abstract}

    In this paper, we prove the existence of global solutions in  $H^3(\mathbb{R})\cap H^{2,1}(\mathbb{R})$ to the Fokas-Lenells (FL) equation on
     the line when the initial data includes solitons.
  A key tool in proving this result   is  a newly modified Darboux transformation, which adds or subtracts
      a soliton with given spectral and scattering parameters.
      In this way   the inverse scattering transform technique is then applied to establish  the global well-posedness of
        initial value problem  with a finite number of solitons based on our previous results on the global well-posedness of the FL equation.

\end{abstract}

\end{frontmatter}
\tableofcontents

    \section{Introduction}
  The present paper is concerned with   the existence  of global solutions to  the Cauchy problem  for   the Fokas-Lenells (FL) equation on the line
  \begin{align}
    & u_{tx}+ u -2iu_x - u_{xx}- |u|^2u_x=0 \label{cs} \\
    &u(x,t)|_{t=0} = u_0(x).\label{cs1}
 \end{align}
The FL equation is an integrable generalization of the nonlinear Schr\"{o}dinger (NLS) equation \cite{Fokas1}.
It   was  also used as a  model for the femtosecond pulse propagation through single mode optical silica fiber \cite{Hos1,Hos2,Hos3}.
In recent years, some interesting mathematical structure and exact solutions to the FL equation have been studied by using various  methods \cite{dfa37, oan36, adm38,adm39,zhang2021, shc41, rwo34,Lenells2009,xiaofan,zhao2013,xujian2,zf1}. For the Schwartz initial data  $u_0(x)\in S(\mathbb{R})$,
  we obtained the long-time asymptotics for the solution  to  the Cauchy problem (\ref{cs})-(\ref{cs1}) via Deift-Zhou steepest descent method  \cite{Lta}.
 For the   weighted Sobolev initial data   $u_0(x)\in H^{3,3}(\mathbb{R})$,   we  presented
  long-time behaviors of the solution  to  the Cauchy problem (\ref{cs})-(\ref{cs1})  by using $\bar\partial$-steepest descent method  \cite{ cqy}.
It is well-known that the existence of a global solutions or the well-posedness of the initial value problem of a partial
 differential equation is the theoretical guarantee to study  the long-time asymptotics.
The  global well-posedness of the periodic initial value problem for the FL equation   was proved by  Fokas and  Himonas  \cite{wpo10}.
A natural question is whether a global solution to  the FL equation on the line for appropriate  initial data $u_0(x)$ exists.

 Recently,  for the weighted Sobolev initial value     $u_{0}  \in H^3(\mathbb{R}) \cap H^{2,1}(\mathbb{R})$  including    no eigenvalues or resonances,
  we have proved  that there exists a unique global solution to  the Cauchy problem (\ref{cs})-(\ref{cs1}) of the FL equation    \cite{laff2}
 $$u  \in C\left([0, \infty); H^3(\mathbb{R}) \cap H^{2,1}(\mathbb{R})\right). $$
 Furthermore, the map
    \begin{equation}
    H^3(\mathbb{R}) \cap H^{2,1}(\mathbb{R}) \ni u_{0} \mapsto u \in C\left([0, \infty); H^3(\mathbb{R}) \cap H^{2,1}(\mathbb{R})\right)
    \end{equation}
    is Lipschitz continuous.
   The   assumption on the initial data $u_0(x)$ is  reasonable  due to   the fact that  if the norm
    \begin{equation}
    2\|u_{0,x}\|^2_{L^2}+\|u_{0,x}\|^3_{L^3}+2\|u_{0,xx}\|_{L^1}+\|u_{0,x}\|_{L^1}<1,\label{xyq3}
    \end{equation}
     then the scattering coefficient $a(k)$ has  no eigenvalues or resonances.
     Our  main technical  tool is the inverse scattering  transform method based on the representation
    of a Riemann-Hilbert (RH) problem associated with the above Cauchy problem.
 This technique was  first applied to prove the global well-posedness of the derivative NLS equation \cite{dep}.
Recently, we also used it to prove the  existence of global solutions to the nonlocal Schrodinger equation on the line  \cite{zf2022}.

\begin{figure}[H]
 \centering
  \begin{tikzpicture}[node distance=2cm]
 \coordinate (A) at (-4.3, 1.2);
 \fill (A) circle (0pt) node[left] {\footnotesize  $u(0,x )\in \mathcal{Z}_N$};
  \draw[ ->](-4.3, 1.2)--(-2.7, 1.2);
 \coordinate (B) at (-2.6, 1.2);
 \fill (B) circle (0pt) node[right] {\footnotesize  $u^{(1)}(0,x)\in \mathcal{Z}_{N-1}$};
  \draw[ ->](0.6, 1.2)--(1.6, 1.2);
   \draw(2, 1.2)node {$\cdots$};
    \draw[ ->](2.6, 1.2)--(3.6, 1.2);
   \coordinate (C) at (3.6, 1.2);
 \fill (C) circle (0pt) node[right] {\footnotesize  $u^{(N)}(0,x)\in \mathcal{Z}_0$};
 \coordinate (D) at (-4.3, -1);
 \fill (D) circle (0pt) node[left] {\footnotesize  $u(t, x )\in \mathcal{Z}_N$};
  \draw[ <-](-4.3, -1)--(-2.7,-1);
 \draw(-3.5, 1.7)node[below]{\scriptsize   {$DT$} };
 \coordinate (F) at (-2.6, -1);
 \fill (F) circle (0pt) node[right] {\footnotesize  $u^{(1)}(t, x)\in \mathcal{Z}_{N-1}$};
  \draw[  <-](0.6, -1)--(1.6, -1);
   \draw(1.1, 1.7)node[below]{\scriptsize     {$DT$} };
      \draw(3.1, 1.7)node[below]{\scriptsize    {$DT$} };
         \draw(1.1, -0.5)node[below]{\scriptsize     {$DT$} };
      \draw(3.1, -0.5)node[below]{\scriptsize     {$DT$} };
   \draw(2, -1)node {$\cdots$};
    \draw[  <-](2.6, -1)--(3.6, -1);
   \coordinate (G) at (3.6, -1);
 \fill (G) circle (0pt) node[right] {\footnotesize  $u^{(N)}(t, x)\in \mathcal{Z}_0$};
  \draw(-3.5,-0.5)node[below]{\scriptsize    {$DT$} };
\draw[ ->](-5.5, 0.9)--(-5.5, -0.6);
  \draw[ ->](4.5, 0.9)--(4.5, -0.6);
  \draw(0.9,  -2.1)node[right]{\footnotesize  {$\text{ global solution without solitons}$} };
\draw(-6.5, -2.1)node[right]{ \footnotesize  {$\text{  global solution  with solitons}$} };
   \draw[ ->](-5.5, -1.3)--(-5.5, -1.8);
      \draw[ ->](4.5, -1.3)--(4.5, -1.8);
  \end{tikzpicture}
\caption{ The  zeros of the scattering data $a(k)$  and the potential $u$ can be removed   or added by using
an appropriate  Darboux transformation. }\label{fiey}
\end{figure}
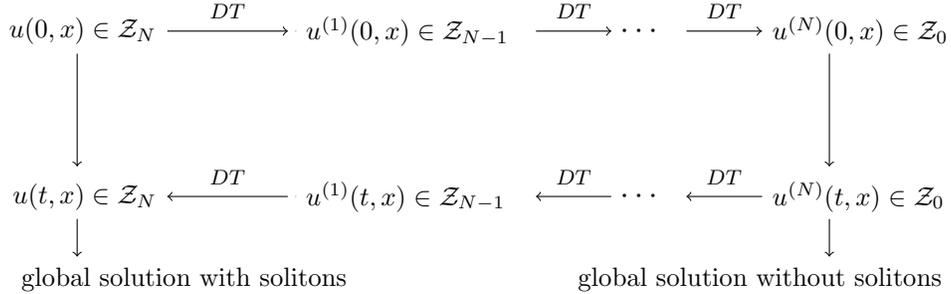

 The goal of our present work is to extend our previous results  on existence of global solutions \cite{laff2}  to the case
that the initial data $u_0  \in H^3(\mathbb{R}) \cap H^{2,1}(\mathbb{R})  $   may allow the  appearance of  zeros of the scattering coefficient $a(k)$.
The main  results are enunciated in the following theorem.
\begin{theorem}  \label{74}
For every $u_0 \in H^3(\mathbb{R}) \cap H^{2,1}(\mathbb{R})$ such that the spectral problem (\ref{cslp})
admits no resonances, there exists a unique global solution $u(t,\cdot) \in H^3(\mathbb{R}) \cap$ $H^{2,1}(\mathbb{R})$ of
the Cauchy problem (\ref{cs})-(\ref{cs1}) for every $t\in\mathbb{R}$.
\end{theorem}
 This result is achieved by construction of    a new invertible   Darboux transformation (DT) which removes
 these  zeros  of   $a(k)$ so that we can apply the previous results \cite{laff2}.
The proof  scheme  can be described as the diagram in Figure \ref{fiey}, in which we denote the space
    \begin{equation}
  \mathcal{ Z}_N=\{ u \in H^3(\mathbb{R})\cap H^{2,1}(\mathbb{R}),   a(k_j)=0, \ k_j\in \mathbb{C}_I, j=1,\cdots, N  \}\label{spacef}
    \end{equation}
  The   Darboux and  B\"{a}cklund transformations have  been successfully applied in the NLS equation and derivative NLS equation
 to prove  existence of global solutions with solitons  or asymptotic  stability of solitons  \cite{Deift,pp2017,Cuccagna}.

Indeed there  is some   work on the Darboux transformations  of FL equation to construct soliton solutions or rogue wave solutions \cite{lra40,grw2019,jsh2012,yw2020}.
 However, these classical Darboux transformations do not involve the key components of inverse scattering theory such as the analyticity of Jost functions,
 the distribution of zeros of the scattering data  $a(k)$.  Moreover, considering our special requirement, our  desired  Darboux transformations should be
invertible  and has no singularity with respect to the variable $x\in \mathbb{R}$, which can therefore be used to add  or subtract zeros of  scattering data $a(k)$,
and preserves the original and the new   potentials in the same Sobolev space $H^3(\mathbb{R}) \cap H^{2,1}(\mathbb{R})$.
For this purpose,
 we construct   a kind of   modified   two-fold  Darboux transformation.  A little different from  those \cite{lra40,grw2019,jsh2012,yw2020},
 our   Darboux matrix  is in the  term of the negative power  for  the spectral  parameter $k$.
 The invertibility of the  Darboux transformation helps us establish a bijection between the $0$-soliton solution and the $1$-soliton solution.
 Therefore, according to the procedure of Figure \ref{fiey}, we obtain the existence of global solutions to the Cauchy problem (\ref{cs})-(\ref{cs1}) with finite solitons.

     The structure of the paper is as follows. In  Section \ref{sec:section2}, we quickly  review   some main results on the direct scattering transform detailed in \cite{laff2}.
   In Section \ref{sec:section3},  we present  a new two-fold Darboux transformation which is related to  direct scattering transform and  suitable for our subsequent analysis.
    The invertibility of the   Darboux transformation is  further shown. In  Section \ref{sec3},
     we concentrate on the properties of  scattering data  and Jost functions  under the action of the Darboux transformation. We prove that both
     original and new   potentials   belongs to the same Sobolev space  $H^3(\mathbb{R})\cap H^{2,1}(\mathbb{R})$.
      In Section \ref{sec4}, we  consider the time evolution of the   Darboux transformation
       and obtain  the existence of   global solutions in the space  $H^3(\mathbb{R})\cap H^{2,1}(\mathbb{R})$ to FL equation on the  line  when the initial data $u_0$ includes  solitons.

    \section{Review of  the direct scattering transform   }
    \label{sec:section2}
 In this paper, we define and use the following  weighted Sobolev spaces
    \begin{align}
    &L^{p,s}(\mathbb{R}):=\{ u \in L^{p}(\mathbb{R}): \ \   \langle x \rangle^s u  \in L^{p}(\mathbb{R})  \}, \nonumber\\
    &H^{k,s}(\mathbb{R})=\left\{u  \in L^{2,s}(\mathbb{R}): \ \  \partial^j_xu   \in L^{2,s}(\mathbb{R}),\quad j=1,\cdots,k \right\},\nonumber
    \end{align}
    where $ \langle x \rangle := \sqrt{1+|x|^2}$.

  \subsection{Existence of Jost functions}
    \label{sub21}
\hspace*{\parindent}

    We review some main results on the direct scattering transform associated with the Cauchy problem  (\ref{cs})-(\ref{cs1})   and the existence   of its Jost solutions.
    For details, please see  \cite{laff2}.
    The FL equation (\ref{cs}) admits a  Lax pair \cite{dfa37,oan36}
    \begin{align}
    &\phi_x+ik^2\sigma_3\phi=kP_x\phi,\label{cslp}\\
    &\phi_t+i\eta^2\sigma_3\phi=H\phi,\label{cslp2}
    \end{align}
    where
    \begin{align}
    & \sigma_3=\begin{pmatrix}
    1&0\\
    0&-1
    \end{pmatrix}, \ \ \ \  P=\begin{pmatrix}
    0&u\\
    -\bar{u}&0
    \end{pmatrix},\label{qq}\\
    &
    \eta=\left(k-\frac{1}{2k}\right),\quad H= kU_x+\frac{i}{2}\sigma_3\left(\frac{1}{k}P-P^2\right). \nonumber
    \end{align}
    By making  a transformation
    \begin{equation}
    \psi(x,t;k)=\phi(x,t;k)e^{i(k^2x+\eta^2t)\sigma_3},\label{b3}
    \end{equation}
    the Lax pair (\ref{cslp})-(\ref{cslp2}) is changed to
    \begin{align}
    &\psi_x+ik^2[\sigma_3,\psi]=kP_x\psi,\label{laxn1}\\
    &\psi_t+i\eta^2[\sigma_3,\psi]=H\psi.\label{laxn2}
    \end{align}
    This  Lax pair admits   the Jost functions with asymptotics
    \begin{equation}
    \psi_{\pm}(x,t;k)\sim I, \quad  x\rightarrow \pm \infty,\label{j1j}
    \end{equation}
    which  satisfy    Volterra  integral equations
    \begin{equation}
    \psi_{\pm}(x,t;k)=I+k\int_{ \pm \infty }^{ x }e^{-2ik^2(x-y)\widehat{\sigma}_3}
    P_y(y)\psi_{\pm}(y,t;k) dy.\label{fi5}
    \end{equation}

    \begin{proposition} \label{proposition1}
    Let $u_0(x)\in H^3(\mathbb{R})\cap H^{2,1}(\mathbb{R})$, and denote
    $\psi_{\pm}(x,t;k)=\left(\psi_{\pm,1} (x,t;k),\psi_{\pm,2}(x,t;k)\right) $
    with the scripts  $1$ and $2$ denoting the first and second columns of $\psi_{\pm}(x,t;k)$. Then we have
    \begin{itemize}

    \item[$\blacktriangleright$] {\bf Analyticity}:   The integral equation (\ref{fi5}) would admit a    unique  solution
      $\psi_{\pm}(x,t;k)$.  Moreover, $ \psi_{-,1}(x,t;k)$, $ \psi_{+,2}(x,t;k)$   are analytical in the domain $ D_+;$  and   $ \psi_{+,1}(x,t;k)$ and $\psi_{-,2}(x,t;k)$ are analytical  in
      in the domain $D_-$,
      where
    \begin{equation}
    D_+=\{ k:  {\rm Im} k^2>0\}, \  \ D_-=\{ k:  {\rm Im} k^2<0\};
    \end{equation}

    \item[$\blacktriangleright$] {\bf Symmetry}:     $ \psi_{\pm}(x,t;k)$    satisfy   the symmetry relations
    \begin{align}
    &\psi_{\pm}(x,t;k) =\sigma_2 \overline{\psi_{\pm}(x,t;\bar {k} )} \sigma_2,\ \ \ \psi_{\pm}(x,t;k) =\sigma_3  \psi_{\pm}(x,t;-k) \sigma_3,\label{dcx}
    \end{align}
    \item[$\blacktriangleright$] {\bf Asymptotics}:   $ \psi_{\pm}(x,t;k)$   have asymptotic properties
    \begin{align}
    &\psi_{\pm}(x,t;k)=\mathrm{e}^{ -ic_\pm(x)\sigma_3}+\mathcal{O}(k^{-1}),\quad k\rightarrow\infty,\label{jjw}
    \end{align}
    where
    \begin{equation}
    c_\pm(x)=\frac{1}{2}\int_{\pm\infty}^x|u_y(y,t)|^2dy. \label{c1c2}
    \end{equation}

    \item[$\blacktriangleright$] {\bf Boundedness}: For every $k$ with $\operatorname{Im}(k^2)>0$ and for all $u$ satisfying $\|u_x\|_{L^1 \cap L^{\infty}}+\left\|\partial^2_x u\right\|_{L^1} \leq M$ there exists a constant $C_M$ which does not depend on $u$, such that
   \begin{equation}
    \left\|\psi_{\pm}(\cdot ;k)\right\|_{L^{\infty}} \leq C_M.\label{psib}
   \end{equation}
    \end{itemize}
    \end{proposition}

    In the following sections, we first consider   the partial spectral problem  (\ref{cslp})  with $t$ being  a parameter,  so we  omit the variable $t$ as usual. For example, $\psi(x,t;k)$ is just written as $\psi(x;k)$.  We will discuss the  time evolution of scattering data and  reconstruct the  potential $u(x,t)$ with $t$  in Section \ref{sec41}.

    Since  $\psi_{\pm}(x;k)e^{-ik^2x \sigma_3}$ are two solutions to the spectral problem (\ref{cslp}),
    they are   linearly dependent and satisfy the scattering relation
   \begin{equation}
   \psi_-(x;k)=\psi_+(x;k)e^{-ik^2x\widehat{\sigma}_3}S(k),\label{phstgx}
   \end{equation}
   where $S(k)$ is a scattering matrix given by
   \begin{equation}
   S(k)=\begin{pmatrix}
   a(k) & b(k) \\[3pt]
   -\overline{b(\bar{k})} & \overline{a(\bar{k})}
   \end{pmatrix}, \ \ \det S(k)=1.\label{xxgx}
   \end{equation}
   By  Proposition \ref{proposition1} and (\ref{phstgx}), $S(k)$ admits the following  symmetries
   \begin{align}
   &S(k) =\sigma_2 \overline{S( \bar {k} )} \sigma_2,\ \ \ S(k) =\sigma_3  S(-k) \sigma_3,  \label{dcx3}
   \end{align}
    and asymptotics
   \begin{align}
   &S(k)=I+\mathcal{O}(k^{-1}),\quad k\rightarrow\infty.
   \end{align}

    It is worth to mention that we would use the vector form $\varphi_{\pm}(x;k)$ and $\phi_\pm(x;k)$ which are the first and second columns respectively of the Jost function $\psi_\pm(x;k)$ in the following calculation and analysis, as analysis on the vector form preserves the same properties as on the matrix form and can make the calculation more concise. We need to convert the matrix results of the relevant Jost functions and scattering coefficients into the vector cases for use.
    \begin{proposition}
      Let $u \in H^3(\mathbb{R}) \cap H^{2,1}(\mathbb{R})$, for every $k \in$ $\mathbb{R} \cup i \mathbb{R}$, there exist unique solutions $\varphi_{\pm}(x ; k) e^{-i k^2 x}$ and $\phi_{\pm}(x ; k) e^{i k^2 x}$ to the spectral problem (\ref{laxn1}) with $\varphi_{\pm}(\cdot ; k) \in L^{\infty}(\mathbb{R})$ and $\phi_{\pm}(\cdot ; k) \in L^{\infty}(\mathbb{R})$ such that
      \begin{itemize}

        \item[$\blacktriangleright$] {\bf Asymptotics}:
        \begin{equation}
          \varphi_{\pm}(x ; k) \rightarrow e_1,
          \quad \phi_{\pm}(x ; k) \rightarrow e_2, \quad \text { as } \quad x \rightarrow \pm \infty;\label{vpjj}
        \end{equation}

        \item[$\blacktriangleright$] {\bf Scattering Data}: The scattering relation (\ref{phstgx}), the scattering data $a(k)$ and $b(k)$  can be expressed in term of the determinant
    \begin{align}
      &a(k)=\det(\varphi_-(x ; k) e^{-i k^2 x}, \phi_+(x ; k) e^{i k^2 x}),\label{ak} \\
      &b(k)=\det(\varphi_+(x ; k) e^{-i k^2 x}, \varphi_-(x ; k) e^{-i k^2 x}).\label{bk}
      \end{align}

    \end{itemize}
    \end{proposition}

  Then, we have

   \begin{proposition}  The scattering data $a(k)$ and $b(k)$ are even and odd functions
     respectively and admit  the following  properties
   \begin{itemize}

   \item[$\blacktriangleright$] {\bf Symmetries}:
   \begin{equation}
   a(-k)=a(k), \quad k \in \overline{D}_+, \quad b(-k)=-b(k), \quad \operatorname{Im} k^{2}=0;\label{abdc}
   \end{equation}

   \item[$\blacktriangleright$] {\bf  Asymptotics}:
   \begin{align}
   &a(k)=  e^{-ic} +\mathcal{O}\left(k^{-1}\right),\quad    \ \
    b(k)=\mathcal{O}\left(k^{-1}\right), \quad k \rightarrow \infty, \nonumber\\
    &a(k)=e^{-ic}\left(1+\mathcal{O}\left(k^2\right)\right),\quad    \ \
    b(k)=\mathcal{O}\left(k^3\right),\quad k\rightarrow 0,\label{aj0}
   \end{align}
   where
   \begin{equation}
   c=c_-(x) - c_+(x) = \frac{1}{2}\int_{-\infty}^\infty |u_x |^2dx,\nonumber
   \end{equation}
   and $c_\pm (x)$ are  defined by  (\ref{c1c2}).

   \end{itemize}
   \end{proposition}

  \subsection{Regularity  of Jost functions}
    \label{sub22}
\hspace*{\parindent}

 \begin{proposition}
  \label{p52}
  For every $u \in H^3(\mathbb{R}) \cap H^{2,1}(\mathbb{R})$ satisfying $\|u\|_{H^3(\mathbb{R}) \cap H^{2,1}(\mathbb{R})} \leq M$ for some $M>0$, let $\varphi_{\pm}(x ; k) e^{-i k^2 x}$ and $\phi_{\pm}(x ; k) e^{i k^2x}$ be Jost functions of the spectral problem (\ref{laxn1}). Fix $k_1 \in \mathbb{C}$ satisfying $\operatorname{Im}\left(k_1^2\right)>$ 0 and denote $\varphi_-:=\varphi_-\left(\cdot ; k_1\right)=\left(\varphi_{-,1}, \varphi_{-,2}\right)^T$ and $\phi_+:=\phi_+\left(\cdot ; k_1\right)=\left(\phi_{+,1}, \phi_{+,2}\right)^T$. Then,
  \begin{align}
  &\left\|\langle x\rangle \varphi_{-,2}\right\|_{L^2(\mathbb{R})}+\left\|\langle x\rangle \partial_x \varphi_-\right\|_{L^2(\mathbb{R})}+\left\|\langle x\rangle \partial_x^2 \varphi_-\right\|_{L^2(\mathbb{R})}\nonumber\\
  &\qquad+\left\|\partial_x^3 \varphi_-\right\|_{L^2(\mathbb{R})} \leq C_M,\label{xv1}\\
  &\left\|\langle x\rangle \phi_{+, 1}\right\|_{L^2(\mathbb{R})}+\left\|\langle x\rangle \partial_x \phi_+\right\|_{L^2(\mathbb{R})}+\left\|\langle x\rangle \partial_x^2 \phi_+\right\|_{L^2(\mathbb{R})}\nonumber\\
  &\qquad+\left\|\partial_x^3 \phi_+\right\|_{L^2(\mathbb{R})} \leq C_M,\label{xv2}
  \end{align}
  where the constant $C_M$ does not depend on $u$.
  \end{proposition}
  \begin{proof}
It is suffice to prove the statement for $\varphi_-$, as the procedure for $\phi_+$ is analogous. From (\ref{psib}), $\varphi_-\in L^{\infty}(\mathbb{R})$ is obvious. First, our aim is to prove $\varphi_{-,2}$ belongs to $L^2(\mathbb{R})$. Under the result that the existence of Jost functions is uniformly in $k$ (for  details, please see in \cite[Proposition 3.1]{laff2}), it is convenient to prove the case for $k=k_1$. Using the integral equation for $\varphi_-$:
  \begin{equation}
  \varphi_-=e_1+K \varphi_-,
  \end{equation}
  where the operator $K$ is given by
  \begin{equation}
  K \varphi_-=k_1 \int_{-\infty}^x\begin{pmatrix}
  1 & 0 \\
  0 & e^{2ik_1^2(x-y)}
  \end{pmatrix} P_y(y) \varphi_-(y;k) dy,
  \end{equation}
  which means
  \begin{equation}
  \varphi_{-,2}(x;k)=k_1\int_{-\infty}^x
  -\bar{u}_xe^{2ik_1^2(x-y)}\varphi_{-,1}(y;k)
  dy,
  \end{equation}
  we deduce
  \begin{equation}
  \left\|\varphi_{-,2}(x;k)\right\|_{L^2\left(-\infty, x_0\right)}\leq \frac{|k_1|\|u_x\|_{L^2\left(-\infty,x_0\right)}}{4\operatorname{Im}(k_1^2)}\left\|\varphi_{-,1}\right\|_{L^{\infty}\left(-\infty, x_0\right)}.
  \end{equation}
  As $u_x \in L^2(\mathbb{R})$, we can divide $\mathbb{R}$ into finite sub-intervals such that $K$ is a contraction as shown above within each sub-interval. By gluing solutions together, we have $\varphi_{-,2}\in L^2(\mathbb{R})$ and
  \begin{equation}
  \left\|\varphi_{-,2}\right\|_{L^2(\mathbb{R})} \leq C_M\|u_x\|_{L^2(\mathbb{R})},\nonumber
  \end{equation}
  where $C_M$ does not depend on $u_x$.
  Next, it follows directly from the (\ref{laxn1}) that
  \begin{equation}
  \partial_x \varphi_{-,1}=k_1 u \varphi_{-,2} \Longrightarrow \partial_x \varphi_{-,1} \in L^2(\mathbb{R}),
  \end{equation}
  and
  \begin{equation}
  \partial_x \varphi_{-,2}=-k_1 \bar{u}_x\varphi_{-,1}+2 i k_1^2 \varphi_{-,2} \Longrightarrow \partial_x \varphi_{-,2} \in L^2(\mathbb{R}).
  \end{equation}
  Differentiating (\ref{laxn1}) once and twice, we also obtain $\partial_x^2 \varphi_-, \partial_x^3 \varphi_- \in L^2(\mathbb{R})$.

 To show $x \varphi_{-,2} \in L^2(\mathbb{R})$, we write
  \begin{equation}
  \partial_x\left(x \varphi_{-,2}\right)=\varphi_{-,2}+x \partial_x \varphi_{-,2}.
  \end{equation}
  Also, with the second component of  (\ref{laxn1}), we obtain
  \begin{equation}
  \partial_x\left(x \varphi_{-,2}\right)=2 i k_1^2 x \varphi_{-,2}+\varphi_{-,2}-kx \bar{u}_x \varphi_{-,2},
  \end{equation}
  the integral expression of which is
  \begin{equation}
  x \varphi_{-,2}(x)=\int_{-\infty}^x e^{2 i k_1^2(x-y)} \varphi_{-,2}(y) d y-k_1 \int_{-\infty}^x e^{2 i k_1^2(x-y)} y \bar{u}_y(y) \varphi_{-,2}(y) d y . \nonumber
  \end{equation}
  As each component is bounded in $L^2(\mathbb{R})$ in the right-hand side, we have $x \varphi_{-,2}\in L^2(\mathbb{R})$. Then, it follows from the system (\ref{laxn1}) and its derivative that $x \partial_x \varphi_-, x \partial_x^2 \varphi_- \in L^2(\mathbb{R})$. Combining the above results, we have the bounds (\ref{xv1}) for $\varphi_-$.
  \end{proof}

   \section{A modified Darboux transformation   }
 \label{sec:section3}

    \subsection{Construction of  Darboux transformation }
    \label{sub22}
    \hspace*{\parindent}

    We need to construct a kind of  two-fold Darboux transformation to obtain a 0-soliton solution $u^{(1)}$ from  a 1-soliton solution $u$  to remove   the zeros of $a(k)$.

 According to the characteristic  of  the Lax pair (\ref{laxn1})-(\ref{laxn2}) for the FL equation,  we   consider  a   general   two-fold Darboux transformation    expressed by
\begin{equation}
  \psi^{(1)}(x;k)  =T  (  k) \psi(x;k),\label{2dt1}
\end{equation}
where
\begin{equation}
 T(k)  =\begin{pmatrix}
    f_1&0\\
    0&f_2
  \end{pmatrix}k ^{-2}+
  \begin{pmatrix}
    0&g_1\\
    g_2&0
  \end{pmatrix}k^{-1}+
  \begin{pmatrix}
    h_1&0\\
    0&h_2
  \end{pmatrix}, \nonumber
\end{equation}
and $f_1,f_2,g_1,g_2,h_1,h_2$ are unknown functions to be determined.

Let  $\eta=(\eta_1,\eta_2)^T $ and $\xi=(\xi_1,\xi_2)^T $ be two    vector  solutions  of the Lax pair  (\ref{laxn1})-(\ref{laxn2})
associated with the spectral parameters  $k_1, \ k_2  \in \mathbb{C} \backslash\{0\} $ respectively. The two-fold Darboux transformation is characterized by a two-dimensional kernel
\begin{align}
 & T (k_1) \eta=0,\ \  T(k_2)\xi=0,
\end{align}
which lead to two  linear systems
\begin{align}
  &\begin{pmatrix}
    k_1^{-2}\eta_2&k_1^{-1}\eta_1\\[4pt]
    k_2^{-2}\xi_2&k_2^{-1}\xi_1
  \end{pmatrix}
  \begin{pmatrix}
    f_1\\[4pt] g_1
  \end{pmatrix}=
  \begin{pmatrix}
    -h_1\eta_2\\[4pt] -h_2\xi_2
  \end{pmatrix},\nonumber\\[6pt]
  &\begin{pmatrix}
    k_1^{-2}\eta_2&k_1^{-1}\eta_1\\[4pt]
    k_2^{-2}\xi_2&k_2^{-1}\xi_1
  \end{pmatrix}
  \begin{pmatrix}
    f_2\\[4pt] g_2
  \end{pmatrix}=
  \begin{pmatrix}
    -h_1\eta_1\\[4pt] -h_2\xi_1
  \end{pmatrix}. \nonumber
\end{align}
Specially, taking  $h_1=h_2=-1$,   we  obtain
\begin{align}
&f_1=\frac{k_1^2k_2\eta_2\xi_1-k_1k_2^2\eta_1\xi_2}{k_2\eta_2\xi_1-k_1\eta_1\xi_2}, \ \
 f_2=\frac{k_1^2k_2\eta_1\xi_2-k_1k_2^2\eta_2\xi_1}{k_2\eta_1\xi_2-k_1\eta_2\xi_1},\label{f2}
\end{align}
and
\begin{align}
  &g_1=\frac{k_2^2\eta_2\xi_2-k_1^2\eta_2\xi_2}{k_2\eta_2\xi_1-k_1\eta_1\xi_2}, \ \
 g_2=\frac{k_2^2\eta_1\xi_1-k_1^2\eta_1\xi_1}{k_2\eta_1\xi_2-k_1\eta_2\xi_1}.\label{g2}
  \end{align}
Substituting  (\ref{2dt1}) into the Lax pair  (\ref{laxn1})-(\ref{laxn2}) yields
\begin{align}
&u^{(1)}(x,t)=\frac{f_1}{f_2}u(x,t)+\frac{g_1}{f_2}:=D (\psi;k)u(x,t),\label{2dt2}\\
&u_x^{(1)}(x,t)=u_x(x,t).\label{2dt3}
\end{align}
Therefore, (\ref{2dt1}) together with  (\ref{2dt2})-(\ref{2dt3}) give the desired  two-fold Darboux transformation,
in which  the relation between old and new components is given by
\begin{align}
&\{ u, \psi\}\longrightarrow \{ u^{(1)}, \psi^{(1)}\}. \label{2dte3}
\end{align}

 To make  above Darboux transformation (\ref{2dte3}) be able to use for us,
 we need to make some modification to it.

  \subsection{Modification to the Darboux transformation  }
 \hspace*{\parindent}

\noindent {\bf (A) odd or subtract zeros:}

For a fixed   parameter  $k \in \mathbb{C}$,  and $\eta =\left(\eta_1, \eta_2\right)^T, \ \xi=\left(\xi_{1}, \xi_2\right)^T\ \in\mathbb{C}^2$,
we define a bilinear form
$$m_{k}(\eta, \xi):=k  \eta_{1} \bar{\xi}_{1}+\bar k \eta_{2} \bar{\xi}_{2}. $$

According to the symmetry (\ref{dcx3}) of scattering data,
if $\pm k_1$  is a zero of $a(k)$ in the first and	third quadrants,
then $  \pm \bar k_1$ are also zeros of
$  \overline{a(\bar k)}$ in the second  and	forth  quadrants.
To remove these zeros,  by  specially taking  $k_2=\bar k_1$, $\xi=\bar \eta $ in (\ref{f2})-(\ref{g2}), we rewrite them as
\begin{align}
 &f_1=|k_1|^2A_{k_1}(\eta),\ \ \
 f_2=|k_1|^2\overline{A_{k_1}}(\eta),\\
 &g_1=C_{k_1}(\eta),\ \ \
 g_1=\overline{C_{k_1}}(\eta),
\end{align}
where
\begin{align}
&A_{k_1}(\eta):=\frac{m_{\bar{k}_1}(\eta,\eta)}{m_{k_1}(\eta, \eta)},\ \ \
 C_{k_1}(\eta):=\left(k^{2}_1-\bar{k}^{2}_1\right) \frac{\eta_1 \bar{\eta}_2}{m_{k_1}(\eta, \eta)}, \ \
\end{align}
The Darboux transformation (\ref{2dt1}) and (\ref{2dt2}) become
\begin{align}
  &\psi^{(1)}(x;k)  =T(\eta, k, k_{1}) \psi(x;k),\\
  &u^{(1)}=A_{k_1}(\eta)(-A_{k_1}(\eta) u+C_{k_1}(\eta)/|k_1|^2),\label{u2bh}
  \end{align}
where the Darboux matrix is given by
\begin{align}
  &T(\eta, k, k_{1}):=\begin{pmatrix}
    |k_1|^2\overline{A_{k_1}(\eta)}k^{-2}-1 & C_{k_1}(\eta)k^{-1} \\
    \overline{C_{k_1}}(\eta)k^{-1} &  |k_1|^2A_{k_1}(\eta)k^{-2}-1
  \end{pmatrix}. \label{u2bh3}
  \end{align}

\noindent {\bf (B)  retain the same  asymptotics:}

We need to  make some modification to the  Darboux matrix (\ref{u2bh3}) such that the boundary conditions
$$ \varphi^{(1)}(x;k) \to e_1, \ \phi^{(1)}(x;k) \to e_2,\ as \ x \to \infty$$
are satisfied.

We choose $\eta=\varphi_-(x;k_1) e^{-i k_1^2x}$, then
\begin{equation}
C_{k_{1}}(\eta) \rightarrow 0, \quad A_{k_{1}}(\eta)=A_{k_{1}}\left(\varphi_-\right) \rightarrow \bar{k}_{1} / k_{1}, \ as \ x \rightarrow-\infty, \label{ajj1}
\end{equation}
and we thus have
\begin{equation}
T(\eta,k,k_1) \rightarrow T_{-\infty}(e_1,k,k_1)=\operatorname{diag}\left(-\frac{k^2-k_{1}^2}{k^2},-\frac{k^2-\bar{k}_1^2}{k^2}\right), x \rightarrow-\infty. \nonumber
\end{equation}
The definition of the new Jost functions $\varphi_-^{(1)}$ and $\phi_-^{(1)}$ are as follows
\begin{equation}
  \begin{aligned}
(\varphi_-^{(1)}, \phi_-^{(1)})&=T(\varphi_-, \phi_-) T_{-\infty}^{(1),-1}\\
&=(T ^{(1)}\varphi_-, T \phi_-)\operatorname{diag}
\left(-\frac{k^2}{k^2-k_{1}^2} , -\frac{k^2}{k^2-\bar{k}_{1}^2}
\right).\nonumber
  \end{aligned}
\end{equation}
Based on  the asymptotics (\ref{vpjj}), it is clear that
\begin{align}
&\varphi_-^{(1)}=-\frac{k^2}{k^2-k_{1}^2} T \varphi_- \sim e_{1},\quad x \rightarrow-\infty,\\
&\phi_-^{(1)}=-\frac{k^2}{k^2-\bar{k}_1^2}T \phi_- \sim e_{2},\quad x \rightarrow-\infty.
\end{align}
On the other hand, if we choose $\eta=\varphi_+(x,k) e^{i k_1^2 x}$, we notice that
\begin{equation}
C_{k_{1}}(\eta) \rightarrow 0, \quad A_{k_{1}}(\eta)=A_{k_{1}}\left(\varphi_-\right) \rightarrow k_{1} / \bar{k}_{1}, \quad x \rightarrow+\infty,
\end{equation}
thus we have
\begin{equation}
  T(k) \rightarrow T_{+\infty}(k)=\operatorname{diag}\left(-\frac{k^2-\bar{k}_1^2}{k^2},-\frac{k^2-k_1^2}{k^2}\right), \quad x \rightarrow+\infty.
\end{equation}
Then
\begin{equation}
  \begin{aligned}
(\varphi_{+}^{(1)}, \phi_+^{(1)})&=T\left(\varphi_+, \phi_+\right) T_{+\infty}^{(1),-1}\\
&=(T \varphi_+, T \phi_+)\operatorname{diag}\left(
  -\frac{k^2}{k^2-\bar{k}_1^2},-\frac{k^2}{k^2-k_{1}^2}
 \right).\nonumber
  \end{aligned}
\end{equation}
Under the asymptotics (\ref{vpjj}), we also have
\begin{align}
&\varphi_{+}^{(1)}=-\frac{k^2}{k^2-\bar{k}_{1}^2}T \varphi_+ \sim e_{1},\quad x \rightarrow+\infty, \\
&\phi_+^{(1)}=-\frac{k^2}{k^2-k_1^2}T \phi_+ \sim e_{2},\quad x \rightarrow+\infty .
\end{align}
Based on the above results, we finally modify the  Darboux matrix (\ref{u2bh3})  as follows
\begin{equation}
T(\eta, k, k_1)=-\frac{1}{k^2-k_1^2}\begin{pmatrix}
  |k_1|^2\overline{A_{k_1}(\eta)}-k^2 & C_{k_1}(\eta)k \\
  \overline{C_{k_{1}}(\eta)}k& |k_1|^2A_{k_1}(\eta)-k^2
\end{pmatrix}.\label{dt2}
\end{equation}
As $A_{k_1}(\eta)$ and $C_{k_1}(\eta)$ are bounded in $x$ for all considered choices of $\eta$, the functions $\varphi_\pm^{(1)}(x ; k)$ and $\phi_\pm^{(1)}(x ; k)$ are bounded functions of $x$ for every $k \in \mathbb{C} \backslash\left\{\pm k_1, \pm \bar{k}_1\right\}$. Also, the asymptotics of $\varphi_\pm^{(1)}(x ; k)$ and $\phi_\pm^{(1)}(x ; k)$ are
\begin{align}
  &\lim_{x\rightarrow \pm \infty}\varphi_\pm^{(1)}=T(e_1,k, k_1) e_1=e_1,\\
  &\lim_{x\rightarrow \pm \infty}\phi_\pm^{(1)}(x ; k)=T(e_2,k, k_1) e_2=e_2.
\end{align}

The invariance of $A_k$ and $C_k$ under a multiplication by a nonzero complex number is shown in the following proposition.

\begin{proposition}
  As for $x \rightarrow \pm \infty$, we note that $A_k(\eta)=A_k(\varphi_-)$ and $C_k(\eta)=C_k(\varphi_-)$ by
  \begin{align}
 &A_k(e_1)=\bar{k}/{k}, \quad A_k(e_2)=k/\bar{k}, \\
 &\overline{A_k(\eta)}=A_{\bar{k}}(\eta), \quad A_k(a \eta)=A_k(\eta),\label{asc}\\
 &C_k(e_1)=C_k(e_2)=0, \quad C_k(a \eta)=C_k(\eta),\label{csc}\\
&A_k(\sigma_3 \eta)=A_k(\eta), \quad A_k(\sigma_1 \eta)=A_{\bar{k}}(\eta),\label{asc1}\\
&C_k(\sigma_3 \eta)=-C_k(\eta), \quad C_k\left(\sigma_1 \eta\right)=\overline{C_k(\eta)}.\label{csc1}
  \end{align}
\end{proposition}
We will show the mapping  $D(\eta,k_1)$ is independent of the choice of the fundamental Jost function   $\varphi_-(x ; k_1)$, $\phi_+(x ; k_1)$, $\varphi_+(x ; \bar{k}_1)$ and $\phi_-(x ; \bar{k}_1)$.
\begin{proposition}
  Assume that $k_1 \in \mathbb{C}_I$ satisfies $a\left(k_1\right)=0$. Given a potential $u \in H^3(\mathbb{R}) \cap H^{2,1}(\mathbb{R})$, it can be seen that
\begin{equation}
\begin{aligned}
u^{(1)}(x) &=D (\varphi_-\left(x; k_1\right) e^{-ik_1^2 x}, k_1) u(x)  =D (\phi_+(x ; k_1) e^{ik_1^2 x},k_1) u(x) \\
&=D (\varphi_+(x ; \bar{k}_1) e^{-i \bar{k}_1^2 x},\bar{k}_1) u(x) =D (\phi_-(x ; \bar{k}_1) e^{i \bar{k}_1^2 x},\bar{k}_1) u(x),\label{u2fj}
\end{aligned}
\end{equation}
where the four Jost functions are the solutions to the  spectral problem (\ref{laxn1}).
\end{proposition}

\begin{proof}
We mainly focus on prove the identity of the four equations in (\ref{u2fj}).
The first one is obtained by setting $\eta$ as $\varphi_-(x ; k_1) e^{-ik_1^2 x}$. Next, combining (\ref{asc}), (\ref{csc}) and $a(k_1)=0$, we apply the linear relation between $\varphi_-(x ; k_1) e^{-ik_1^2 x}$ and $\phi_+(x ; k_1) e^{ik_1^2 x}$ to obtain the second equation. Then, we use the symmetry relation (\ref{dcx}) and properties (\ref{asc1})-(\ref{csc1}) to obtain
\begin{equation}
A_{k_1}(\varphi_-(x ; k_1))=A_{\bar{k}_1}(\sigma_1 \sigma_3 \varphi_-(x ;k_1))=A_{\bar{k}_1}(\overline{\phi_-(x ; \bar{k}_1)})=A_{\bar{k}_1}(\phi_-(x ; \bar{k}_1)),\nonumber
\end{equation}
and
\begin{equation}
C_{k_1}(\varphi_-(x ; k_1))=-\overline{C_{k_1}(\sigma_1 \sigma_3 \varphi_-(x ; k_1)}=-\overline{C_{k_1}(\overline{\phi_-(x ; \bar{k}_1)})}=C_{\bar{k}_1}(\phi_-(x ; \bar{k}_1)).\nonumber
\end{equation}
By invoking the transformation formula (\ref{u2bh}), we prove the third equation. Finally, we derive the fourth equation by
\begin{equation}
\varphi_+(x ; \bar{k}_1) e^{-i \bar{k}_1^2 x}=\overline{\gamma}\phi_-(x ; \bar{k}_1) e^{i \bar{k}_1^2 x}
\end{equation}
as $\overline{a(\bar{k}_1)}=0$.
\end{proof}

After imposing the transformation (\ref{dt2}) on the Jost functions $\varphi_-(x;k)$, $\phi_-(x;k)$, $\varphi_+(x;k)$, and $\phi_+(x;k)$, for $k \in$ $\mathbb{C} \backslash\left\{\pm k_1, \pm \bar{k}_1\right\}$, the new Jost functions becomes
\begin{align}
&\varphi_-^{(1)}(x ; k)=T(\varphi_-\left(x ; k_1\right) e^{-i k_1^2 x}, k, k_1) \varphi_-(x ; k),\label{vx1}\\
&\varphi_{+}^{(1)}(x ; k)=T(\varphi_+(x ; \bar{k}_1) e^{-i \bar{k}_1^2 x}, k, \bar{k}_1) \varphi_+(x ; k),\label{vx2}\\
&\phi_+^{(1)}(x ; k)=T(\phi_+(x ; k_1) e^{i k_1^2 x}, k, k_1) \phi_+(x ; k),\label{vx3}\\
&\phi_-^{(1)}(x ; k)=T(\phi_-(x;\bar{k}_1) e^{i \bar{k}_1^2 x}, k, \bar{k}_1) \phi_-(x;k).\label{vx4}
\end{align}

 Therefore, $\left\{\varphi_\pm^{(1)}(x ; k) e^{-i k^2 x}, \phi_\pm^{(1)}(x ; k) e^{i k^2 x}\right\}$ are Jost functions of the  spectral problem (\ref{laxn1}) associated with the potential $u^{(1)}=D (\eta; k_1) u$.

 \subsection{Invertibility of the Darboux transformation}
 \label{sub23}
  \hspace*{\parindent}

Next, we construct the left inverse of the transformation $D (\eta,k_1)$, which proves the invertibility of the Darboux transformation and finally builds the one-to-one correspondence between the $1$-soliton solution and the $0$-soliton solution.
 \begin{proposition}
 The left inverse for the mapping $D (\eta,k_1)$ exists.
 \end{proposition}
 \begin{proof}
  If we assume the left inverse of $D (\eta;k_1)$ is $D (\widetilde{\eta};k_1)$, then we can construct the following equation
  \begin{align}
   &\widetilde{u}=D (\widetilde{\eta};k_1) D (\eta;k_1)u\nonumber\\
  & =A_{k_1}(\widetilde{\eta})^2 A_{k_1}(\eta)^2 u+A_{k_1}(\widetilde{\eta})\left[-A_{k_1}(\widetilde{\eta}) A_{k_1}(\eta) C_{k_1}(\eta)+C_{k_1}(\widetilde{\eta})\right],
   \end{align}
   where $\widetilde{\eta}$ satisfies
   \begin{equation}
   A_{k_1}(\widetilde{\eta})^2 A_{k_1}(\eta)^2=1,\label{ak1t}
   \end{equation}
   and
  \begin{equation}
   -A_{k_1}(\widetilde{\eta}) A_{k_1}(\eta) C_{k_1}(\eta)+C_{k_1}(\widetilde{\eta})=0.\label{ak2t}
  \end{equation}
 With further analysis, (\ref{ak1t}) and (\ref{ak2t}) are decomposed into one of the following forms:
   \begin{equation}
   \overline{A_{k_1}\left(\widetilde{\eta}\right)}=A_{k_1}(\eta), \quad C_{k_1}\left(\widetilde{\eta}\right)=C_{k_1}(\eta),\label{at1}
   \end{equation}
   or
   \begin{equation}
   \overline{A_{k_1}\left(\widetilde{\eta}\right)}=-A_{k_1}(\eta), \quad C_{k_1}\left(\widetilde{\eta}\right)=-C_{k_1}(\eta) .\label{at2}
   \end{equation}
   Next, we proceed to decide which of the two forms would yield the existence of the left inverse. Based on (\ref{ajj1}), we realize that
   \begin{equation}
   \lim_{x \rightarrow-\infty}|\widetilde{\eta}_1| / |\widetilde{\eta}_2| \neq 0,\label{e1be2}
   \end{equation}
   as it contradicts with the first equation in (\ref{ak2t}) with $k_1 \neq 0$. According to the second term in (\ref{ak2t}), we derive $C_{k_1}(\widetilde{\eta}) \rightarrow 0$ as $x \rightarrow-\infty$ because $C_{k_1}(\eta) \rightarrow 0$ as $x \rightarrow-\infty$. Also, from (\ref{e1be2}), we have
   \begin{equation}
\lim_{x \rightarrow-\infty} C_{k_1}(\widetilde{\eta})= 0\Rightarrow  \lim_{x \rightarrow-\infty}|\widetilde{\eta}_2| /|\widetilde{\eta}_1|= 0,
   \end{equation}
which implies
\begin{equation}
 \lim_{x \rightarrow-\infty}\overline{A_{k_1}(\widetilde{\eta})} =k_1 / \bar{k}_1. \nonumber
\end{equation}
 Meanwhile, from the first equation in (\ref{ak2t}), as $k_1 \in \mathbb{C}_I$, we obtain $\operatorname{Re}(k_1^2)=0$ with $\arg (k_1)=\pi / 4$. Thus, we have $k_1=\left|k_1\right| e^{i \pi / 4}$ and use the first equation in (\ref{ak2t}) to obtain
 \begin{equation}
 |\widetilde{\eta}_1|^2|\eta_2|^2+|\widetilde{\eta}_2|^2|\eta_1|^2=0.\nonumber
 \end{equation}
 This indicates $\widetilde{\eta}=0$ which contradicts with $\widetilde{\eta}\neq 0$. Consequently, we conclude that the choice (\ref{at1}) is reasonable.

 Therefore, we use the choice (\ref{at2}) to define $\widetilde{\eta}$ and satisfy the system (\ref{ak1t})-(\ref{ak2t}). As $k_1 \in \mathbb{C}_I$, the condition $\overline{A_{k_1}(\widetilde{\eta})}=A_{k_1}(\eta)$ is equivalently written as
 \begin{equation}
 \left|\eta_1\right|^2|\widetilde{\eta}_1|^2=\left|\eta_2\right|^2|\widetilde{\eta}_2|^2.
 \end{equation}
 Hence, there exists a positive number $n$ such that
 \begin{equation}
 |\widetilde{\eta}_1|=n|\eta_2|, \quad |\widetilde{\eta}_2|=n|\eta_1|.\label{eg1}
 \end{equation}
 In addition, the condition $C_{k_1}(\widetilde{\eta})=C_{k_1}(\eta)$ yields
 \begin{equation}
 \frac{\eta_1 \bar{\eta}_2}{\widetilde{\eta}_1 \widetilde{\bar{\eta}}_2}=\frac{k_1|\eta_1|^2+\bar{k}_1|\eta_2|^2}{k_1|\widetilde{\eta}_1|^2+\bar{k}_1|\widetilde{\eta}_2|^2}, \label{ebz}
 \end{equation}
By substituting (\ref{eg1}) into (\ref{ebz}), it is reformulated as
 \begin{equation}
 n^2 \frac{\eta_1 \bar{\eta}_2}{\widetilde{\eta}_1 \widetilde{\bar{\eta}}_2}=\frac{k_1|\eta_1|^2+\bar{k}_1|\eta_2|^2}{k_1|\eta_2|^2+\bar{k}_1|\eta_1|^2},\label{eg2}
 \end{equation}
 where the right-hand side is of modulus one. Combining (\ref{eg1}) and (\ref{eg2}), the most general solution of  (\ref{at1}) is as follows
 \begin{equation}
 \widetilde{\eta}_1=n_1 \bar{\eta}_2, \quad \widetilde{\eta}_2=n_2 \bar{\eta}_1,\label{eg3}
 \end{equation}
 where $n_1, n_2 \in \mathbb{C}$ meet the condition $\left|n_1\right|=\left|n_2\right|$. Therefore, $D (\widetilde{\eta}, k_1)$ with $\widetilde{\eta}$ given by (\ref{eg3}) is the left inverse of the transformation $D (\eta, k_1)$.

 Then, we will prove a unique choice for the function $\widetilde{\eta}$ can be given by $\overline{\eta}$. As $\eta=(\eta_1,\eta_2)^T$ is the Jost function of the FL spectral problem (\ref{laxn1}) for $k=k_1$, $\widetilde{\eta}=(\bar{\eta}_2,\bar{\eta}_1)^T$ is accordingly the Jost function of the FL spectral problem (\ref{laxn1}) for $k=-\bar{k}_1$. Thus, we have the following equation of $\widetilde{\eta}$
 \begin{equation}
 \partial_x \widetilde{\eta}=-(i \bar{k}_1^2 \sigma_3+\bar{k}_1 P_x) \widetilde{\eta}.\label{eml}
 \end{equation}
Rewriting the equation as elements of the $\eta$, we have
 \begin{align}
 &\partial_x \bar{\eta}_1=i\bar{k}_1^2\bar{\eta}_1+\bar{k}_1\bar{u}_x\bar{\eta}_2,\nonumber\\
 &\partial_x \bar{\eta}_2=-i\bar{k}_1^2\bar{\eta}_2-\bar{k}_1\bar{u}_x\bar{\eta}_1.\nonumber
 \end{align}
With (\ref{eg3}), we deduce the equation of $\widetilde{\eta}$ as follows
 \begin{equation}
 \partial_x \widetilde{\eta}_1=\partial_x n_1\bar{\eta}_2+n_1\partial_x\bar{\eta}_2=(\partial_x n_1-i\bar{k}_1^2n_1)\bar{\eta}_2-\bar{k}_1u_xn_1\bar{\eta}_1.\label{x1}
 \end{equation}
On the other hand, if $\widetilde{\eta}$ is the solution of the FL spectral problem (\ref{laxn1}) associated with the potential $u_x^{(1)}=u_x$ and $k=k_1$, we have
  \begin{equation}
   \partial_x \widetilde{\eta}_1=-ik_1^2n_1\bar{\eta}_2+k_1u^{(1)}n_2\bar{\eta}_1.\label{x2}
  \end{equation}
 Comparing (\ref{x1}) with (\ref{x2}), we finally reach the specific values of $n_{1,2}$
 \begin{equation}
 n_1=e^{i(\bar{k}_1^2-k_1^2)x},\quad n_2=\frac{\bar{k}_1}{k_1}e^{i(\bar{k}_1^2-k_1^2)x}, \nonumber
 \end{equation}
 which give
 \begin{equation}
 \widetilde{\eta}_1=e^{i(\bar{k}_1^2-k_1^2)x}\bar{\eta}_2,\quad \widetilde{\eta}_2=\frac{\bar{k}_1}{k_1}e^{i(\bar{k}_1^2-k_1^2)x}\bar{\eta}_1.\label{xjt}
 \end{equation}
 Combining the above conditions,  we finally obtain $\widetilde{\eta}$, which satisfies the FL spectral problem (\ref{laxn1}) with the potential $u_x^{(1)}$ and  $k=k_1$.
 \end{proof}

  \section{Action of the  Darboux transformation  }
\label{sec3}
 \subsection{Zeros of the new scattering data }
 \label{sec4}
 \hspace*{\parindent}

 Let $\mathcal{Z}_N$ be a space defined by  (\ref{spacef}). If there exists  $k_1 \in \mathbb{C}_I$ satisfies $a(k_1)=0$, then  $u$ belongs to $\mathcal{Z}_1 \subset H^3(\mathbb{R}) \cap H^{2,1}(\mathbb{R})$. In the following, we prove, under the Darboux transformation, $u^{(1)}$ belongs to $\mathcal{Z}_0 \subset H^3(\mathbb{R}) \cap H^{2,1}(\mathbb{R})$ with $a^{(1)}(k_1)\neq 0$.
\begin{proposition}
  \label{l41}
Let $u \in \mathcal{Z}_1$ and $k_1 \in \mathbb{C}_I$ such that $a(k_1)=0$. If $\eta(x)=\varphi_-(x ; k_1) e^{-i k_1^2 x}$, where $\varphi_-$ is the Jost function of the spectral problem (\ref{laxn1}), then $u^{(1)}=D ( \eta,k_1) u$ belongs to $\mathcal{Z}_0$.
\end{proposition}
\begin{proof}
 We show that if  $a(k)$  has a  simple zero   $k=k_1$,   then after the Darboux transformation,
 $$a^{(1)}(k):=\text{det}\left(\varphi_-^{(1)}(\cdot ; k), \phi_+^{(1)}(\cdot ; k)\right)$$
  has no zero in $\mathbb{C}_I$, where $\varphi_-^{(1)}$ and $\phi_+^{(1)}$ are given by (\ref{vx1}) and (\ref{vx3}). Explicit calculation gives
\begin{equation}
  \begin{aligned}
  &a^{(1)}(k)=\text{det}(\varphi_-^{(1)}(x ; k), \phi_+^{(1)}(x;k))\\
  &= \text{det}\left(T\left[\varphi_-\left(x ; k_1\right) e^{-i k_1^2 x}, k, k_1\right] \varphi_--T\left[\phi_+\left(x ;k_1\right) e^{i k_1^2 x}, k, k_1\right] \phi_+\right) \\
  &= \text{det}\left(T\left[\varphi_-\left(x ; k_1\right), k, k_1\right] \varphi_-(x ; k)-T\left[\varphi_-\left(x ; k_1\right), k, k_1\right] \phi_+(x ; k)\right) \\
  &=-a(k) \operatorname{det}\left(T\left[\varphi_-\left(x ; k_1\right), k, k_1\right]\right)= \frac{k^2-\bar{k}_1^2}{k^2}a(k). \nonumber
  \end{aligned}
  \end{equation}
 As $k_1$ is the only simple zero of $a(k)$ in $\mathbb{C}_I$, $a^{(1)}(k)$ has no zeros for $k$ in $\mathbb{C}_I$.

We also calculate the transformation from $b(k)$ to $b^{(1)}(k)$ as follows:
\begin{equation}
\begin{aligned}
b^{(1)}(k) &=\text{det}\left(e^{-2 i k^2 x} \varphi_{+}^{(1)}(x ; k), \varphi_-^{(1)}(x ; k)\right) \\
&=\text{det}\left(e^{-2 i k^2 x} \varphi_{+}^{(1)}(x ; k), T\left[\varphi_-\left(x ; k_1\right) e^{-i k_1^2 x}, k, k_1\right] \varphi_-(x ; k)\right) \\
&=b(k) \text{det}\left(e_1, T\left[e_2, k, k_1\right] e_2\right)   =-b(k).\nonumber
\end{aligned}
\end{equation}
$b(k)$ adds no new zeros and singularities.
\end{proof}

\subsection{Smoothness  of   new Jost functions}
 \hspace*{\parindent}

\begin{proposition}
  Let $\varphi_{\pm}^{(1)}(x ; k)$ and $\phi_{\pm}^{(1)}(x ; k)$ as defined by (\ref{vx1})-(\ref{vx4}). Then, $k=\pm k_1$ and $k=\pm \bar{k}_1$ are removable singularities in the corresponding domains of analyticity of $\varphi_{\pm}^{(1)}(x ; k)$ and $\phi_{\pm}^{(1)}(x ; k)$.
  \end{proposition}
  \begin{proof}
 It is suffice to consider the first Jost function $\varphi_-^{(1)}(x ;k)$ represent by (\ref{vx1}) and the other cases follow similarly. We denote $\varphi_-=(\varphi_{-,1}, \varphi_{-,2})^T$ and $\varphi_-^{(1)}=(\varphi_{-,1}^{(1)}, \varphi_{-,2}^{(1)})^T$ for the 2-vectors, and obtain for $k \in \mathbb{C}_I \cup \mathbb{C}_{I I I} \backslash\left\{\pm k_1\right\}$
 \begin{equation}
   \begin{aligned}
   \varphi_1^{-,(1)}(x;k)=&-\frac{k^2}{k^2-k_1^2}\left\{\left(\frac{|k_1|^2m_{k_1}(\varphi_-,\varphi_-)}{k^2m_{\bar{k}_1}(\varphi_-,\varphi_-)}-1\right)\varphi_{-,1}+\frac{(k_1^2-\bar{k}_1^2)\varphi_{-,1}|\varphi_{-,2}|^2}{km_{\bar{k}_1}(\varphi_-,\varphi_-)}\right\} \\
   =& \frac{\left(k^2-k_1^2\right) \bar{k}_1\left|\varphi_{-,1}(x;k)\right|^2 \varphi_{-,1}(x;k)+G(x;k)}{(k^2-k_1^2) m_{k_1}(\varphi_-, \varphi_-)},
   \end{aligned}
 \end{equation}
   where
  \begin{equation}
   G(x;k):=(k^2-\bar{k}_1^2) k_1\left|\varphi_{-,2}(x;k)\right|^2 \varphi_{-,1}(k)-k(k_1^2-\bar{k}_1^2) |\varphi_{-,2}(x;k)|^2\varphi_{-,1}(x;k).
  \end{equation}
   It is clear $G(x;k_1)=G(x;-k_1)=0$, as $\varphi_{-,1}(x;k)$ is even in $k$ and $\varphi_{-,2}(x;k)$ is odd in \cite{laff2}. $G$ is analytic in $\mathbb{C}_I \cup \mathbb{C}_{I I I}$ on the basis of Proposition 3.1 \cite{laff2}. Hence, we have $G(x;k)=\left(k^2-k_1^2\right) \widetilde{G}(x;k)$, where $\widetilde{F}$ is analytic in $\mathbb{C}_I \cup \mathbb{C}_{I I I}$. Furthermore, through the calculation
  \begin{equation}
   \varphi_{-,1}^{(1)}(x;k)= \frac{\bar{k}_1\left|\varphi_{-,1}(x;k)\right|^2 \varphi_{-,1}(x;k)+\widetilde{G}(x;k)}{m_{k_1}(\varphi_-, \varphi_-)},
  \end{equation}
   we deduce $\pm k_1$ are removable singularities of $\varphi_{-,1}^{(1)}(x;k)$. It is easy to know  $\pm k_1$ are also removable singularities of $\varphi_{-,2}^{(1)}(x;k)$.
  \end{proof}

Next, we characterize the relation between $\widetilde{\eta}$ and the new Jost functions in the following proposition.
\begin{proposition}
  \label{l55}
Fix $k_1 \in \mathbb{C}_I$ such that $a(k_1)=0$ and $a^{\prime}(k_1) \neq 0$. By (\ref{xjt}), $\widetilde{\eta}$ can be expressed as
\begin{equation}
\widetilde{\eta}(x)=\frac{\gamma\bar{k}_1}{k_1a^{(1)}(k_1)} e^{-i k_1^2 x} \varphi_-^{(1)}(x ; k_1)+\frac{\bar{k}_1}{k_1 a^{(1)}(k_1)} e^{ik_1^2 x} \phi_+^{(1)}(x;k_1),\label{exx1}
\end{equation}
where the new Jost functions $\varphi_-^{(1)}$ and $\phi_+^{(1)}$ are constructed in (\ref{vx1}) and (\ref{vx3}), $\gamma \neq 0$ is the norming constant under $a(k_1)=0$, and $a^{(1)}(k_1) \neq 0$ as in Lemma \ref{l41}.
\end{proposition}
\begin{proof}
We denote $\widetilde{\eta}=(\widetilde{\eta}_1, \widetilde{\eta}_2)^T$ and $\varphi_-=(\varphi_{-,1}, \varphi_{-,2})^T$. Components of $\widetilde{\eta}$ given by (\ref{xjt}) are explicitly expressed by
\begin{align}
  &\widetilde{\eta}_1(x)=e^{i\bar{k}_1^2x}\overline{\varphi_{-,2}(x ; k_1)},\label{w1}\\
  &\widetilde{\eta}_2(x)=e^{i\bar{k}_1^2x}\overline{\varphi_{-,1}(x ; k_1)}.\label{w2}
\end{align}
  As $$\lim _{x \rightarrow-\infty} \varphi_-(x ;k_1)=e_1,$$ we have
\begin{equation}
  \lim _{x \rightarrow-\infty} e^{-i k_1^2 x} \widetilde{\eta}(x)=\frac{\bar{k}_1}{k_1}e_2.\label{lm1}
\end{equation}
When $a(k_1)=0$, the relation
\begin{equation}
  \varphi_-(x ; k_1) e^{-ik_1^2 x}=\gamma \phi_+(x ; k_1) e^{i k_1^2 x} \quad x \in \mathbb{R}
  \end{equation}
follows. Substituting it into (\ref{w1}) and (\ref{w2}), $\widetilde{\eta}$ can be rewritten as:
\begin{align}
  &\widetilde{\eta}_1(x)=\gamma e^{i(\bar{k}_1^2-2k_1^2) x} \overline{\phi_{+,2}(x ; k_1)},\\
  &\widetilde{\eta}_2(x)=\frac{\gamma \bar{k}_1e^{i(\bar{k}_1^2-2k_1^2) x} \overline{\phi_{+, 1}\left(x ;k_1\right)}}{k_1}.
  \end{align}
Since $$\lim _{x \rightarrow+\infty} \phi_+\left(x ; k_1\right)=e_2,$$ then
\begin{equation}
\lim _{x \rightarrow+\infty} e^{i k_1^2 x} \widetilde{\eta}(x)=\frac{\gamma\bar{k}_1}{k_1} e_1,\label{lm2}
\end{equation}
follows.
We know that $\widetilde{\eta}$ is a solution of the FL spectral problem (\ref{laxn1}) associated with the new potential $u^{(1)}$ for $k=k_1$ and  $\varphi_-^{(1)}(x ; k)$ and $\phi_+^{(1)}(x ; k)$ are analytic at $k_1$. Any solution of the second-order system can be expressed by
\begin{equation}
\widetilde{\eta}(x)=c_1 \varphi_-^{(1)}(x ; k_1) e^{-i k_1^2 x}+c_2 \phi_+^{(1)}(x ; k_1) e^{i k_1^2 x},
\end{equation}
where $c_1, c_2$ are independent of $x$. With the boundary conditions (\ref{vpjj}) and the representation (\ref{ak}), we obtain the boundary conditions
\begin{equation}
\lim _{x \rightarrow-\infty} e^{-i k_1^2 x} \widetilde{\eta}(x)=c_2 a^{(1)}(k_1) e_2, \quad \lim _{x \rightarrow+\infty} e^{i k_1^2 x} \widetilde{\eta}(x)=c_1 a^{(1)}(k_1) e_1,\label{lm3}
\end{equation}
where $k_1 \in \mathbb{C}_I$. As $a^{(1)}(k_1) \neq 0$ by Lemma \ref{l41}, from (\ref{lm1})-(\ref{lm3}), $c_1$ and $c_2$ can be expressed as
\begin{equation}
c_1=\frac{\gamma\bar{k}_1}{k_1a^{(1)}(k_1)},\quad c_2=\frac{\bar{k}_1}{k_1a^{(1)}(k_1)},
\end{equation}
which yield the equation (\ref{exx1}).

Instead of the decomposition (\ref{exx1}), we can write
\begin{equation}
\widetilde{\eta}(x):=\varphi_-^{(1)}(x ; k_1) e^{-i k_1^2 x}+\alpha_1 \phi_+^{(1)}(x ; k_1) e^{i k_1^2 x}
\end{equation}
because the Darboux transformation (\ref{u2bh}) is invariant if $\widetilde{\eta}$ is multiplied by a nonzero constant.
\end{proof}

\subsection{Regularity  of    new  potentials }
\label{sec5}
\hspace*{\parindent}

\begin{proposition}
  \label{34}
  Under the same conditions as in Lemma \ref{l55}, for every $u^{(1)} \in H^3(\mathbb{R}) \cap$ $H^{2,1}(\mathbb{R})$ satisfying $\left\|u^{(1)}\right\|_{H^3 \cap H^{2,1}} \leq M$ for some $M>0$, the transformation
  \begin{equation}
  D \left(\widetilde{\eta}, k_1\right) u^{(1)} \in H^3(\mathbb{R}) \cap H^{2,1}(\mathbb{R})
  \end{equation}
  satisfies
  \begin{equation}
  \left\|D (\widetilde{\eta}, k_1) u^{(1)}\right\|_{H^3 \cap H^{2,1}} \leq C_M,\label{lk1e}
  \end{equation}
  where the constant $C_M$ does not depend on $u^{(1)}$.
  \end{proposition}
  \begin{proof}
  From the representation (\ref{ak}), we have
   \begin{equation}
    \begin{aligned}
    \left|a^{(1)}(k_1)\right|=& |\left(\varphi_{-,1}^{(1)}(x ; k_1) e^{-i k_1^2 x}+\alpha_1 \phi_{+, 1}^{(1)}(x ; k_1) e^{i k_1^2 x}\right) \phi_{+,2}^{(1)}(x ; k_1) e^{i k_1^2 x} \\
    -&\left(\varphi_{-,2}^{(1)}(x ; k_1) e^{-i k_1^2 x}+\alpha_1 \phi_{+,2}^{(1)}(x ; k_1) e^{i k_1^2 x}\right) \phi_{+, 1}^{(1)}(x ; k_1) e^{i k_1^2 x}| \\
    \leq &\left\|\phi_+^{(1)}(\cdot ; k_1)\right\|_{L^{\infty}}\left(|e^{i k_1^2 x} \widetilde{\eta}_1(x)|+|e^{i k_1^2 x} \widetilde{\eta}_2(x)|\right).
    \end{aligned}\nonumber
  \end{equation}
    As $a^{(1)}(k_1) \neq 0$ by Lemma \ref{l41} and $\left|m_k(\eta, \eta)\right| \geq\left|\operatorname{Re}\left(k_1\right)\right|\left(\left|\eta_1\right|^2+\left|\eta_2\right|^2\right)$, there is a constant $C_M>0$ independently of $u^{(1)}$ which yields
  \begin{equation}
    \frac{1}{\left|m_{k_1}\left(e^{i k_1^2 x} \widetilde{\eta}(x), e^{i k_1^2 x} \widetilde{\eta}(x)\right)\right|} \leq C_M \quad \text { for all } x \in \mathbb{R}.
  \end{equation}
    By using the same argument, we also obtain
  \begin{equation}
    \left|a^{(1)}(k_1)\right| \leq\left|\alpha_1\right|^{-1}\left\|\varphi_-^{(1)}\left(\cdot ; k_1\right)\right\|_{L^{\infty}}\left(\left|e^{-i k_1^2 x} \widetilde{\eta}_1(x)\right|+\left|e^{-i k_1^2 x} \widetilde{\eta}_2(x)\right|\right),
  \end{equation}
    such that $\frac{1}{\left|m_{k_1}\left(e^{-i k_1^2 x} \widetilde{\eta}(x), e^{-i k_1^2 x} \widetilde{\eta}(x)\right)\right|} \leq C_M \quad$ for all $x \in \mathbb{R}$.

    With the bound
    \begin{equation}
      \begin{aligned}
      \left|\frac{\widetilde{\eta}_1\widetilde{\bar{\eta}}_2}{m_{k_1}(\widetilde{\eta}, \widetilde{\eta})}\right| &\leq \frac{\left|\varphi_1^{-,(1)}\left(x ; k_1\right) \overline{\varphi_{-,2}^{(1)}\left(x ; k_1\right)}\right|}{\left|m_{k_1}\left(e^{i k_1^2 x} \widetilde{\eta}, e^{i k_1^2 x} \widetilde{\eta}\right)\right|}+\left|\alpha_1\right|^2 \frac{\left|\phi_{-, 1}^{(1)}\left(x ; k_1\right) \overline{\phi_{-,2}^{(1)}(x ; k_1)}\right|}{\left|m_{k_1}\left(e^{-i k_1^2 x} \eta^{(1)}, e^{-i k_1^2 x} \widetilde{\eta}\right)\right|} \\
     &+\left|\alpha_1\right| \frac{\left|\varphi_1^{-,(1)}(x ;k_1) \overline{\phi_{-,2}^{(1)}(x ; k_1)}\right|+\mid \phi_{-,1}^{(1)}(x ; k_1) \overline{\varphi_{-,2}^{(1)}(x ; k_1) \mid}}{\left|m_{k_1}\left(\widetilde{\eta}, \widetilde{\eta}\right)\right|},
      \end{aligned}
      \end{equation}
      and the bounds (\ref{xv1})-(\ref{xv2}), we obtain
  \begin{equation}
  \left\|C_{k_1}(\widetilde{\eta}) u^{(1)}\right\|_{L^{2,1}} \leq C_M .
  \end{equation}
 By the same proof of Lemma \ref{53}, it shows that
  \begin{equation}
  \left\|D(\widetilde{\eta},k_1) u^{(1)}\right\|_{L^{2,1}} \leq C_M .
  \end{equation}
The conclusion of $\left\|\partial_x\left(D (\widetilde{\eta},k_1) u^{(1)}\right)\right\|_{L^{2,1}}$ and $\left\|\partial_x^2\left(D (\widetilde{\eta},k_1) u^{(1)}\right)\right\|_{L^2}$ are also from (\ref{xv1})-(\ref{xv2}). Finally, the bound (\ref{lk1e}) is obtained immediately.
  \end{proof}

 \begin{lemma}
  \label{p1}
  Consider the initial value problem of a system of linear ordinary differential equations
    \begin{align}
    &\frac{d Y}{d x}=A(x) Y,\label{a1} \\
    &Y\left(x_0\right)=Y_0, \quad x_0 \in[\alpha, \beta],\label{a2}
    \end{align}
    where $Y=\left(y_1(x), \cdots, y_n(x)\right)^T$, $A(x)$ is a $n\times n$ continuous matrix-valued function,
    then the system (\ref{a1})-(\ref{a2}) has    a   unique zero solution.
 \end{lemma}

  Based on the  regularity  of $\varphi(x;k)$ and $\phi(x;k)$ given by Proposition \ref{p52},
  we immediately obtain the following proposition
  that the transformation (\ref{dt2}) can be defined
  as an operator from $u \in H^3(\mathbb{R}) \cap H^{2,1}(\mathbb{R})$ to $u^{(1)}\in H^3(\mathbb{R}) \cap H^{2,1}(\mathbb{R})$.

\begin{proposition}
  \label{53}
   Fix $k_{1} \in \mathbb{C}_{I}$. Given a potential $u \in H^3(\mathbb{R}) \cap H^{2,1}(\mathbb{R})$, define $\eta(x):=\varphi_-(x ; k_1) e^{-i k_1^2 x}$, where $\varphi_-$ is the Jost function for the  spectral problem (\ref{laxn1}). Then, $u^{(1)}$ belongs to $H^3(\mathbb{R}) \cap H^{2,1}(\mathbb{R})$.
\end{proposition}
\begin{proof}
  The proof is in two parts:

  $\bullet$ $u^{(1)}$ is another smooth solution of the FL equation.

  $\bullet$  $\left\|u^{(1)}\right\|_{H^3\cup H^{2,1}}<\infty$.\\
  We see
  $$
  A_k(\eta)=-1 \quad \text { and } \quad C_k(\eta)=0 \quad \text { if } \quad k \in \mathbb{R} \cup i \mathbb{R},
  $$
  which implies $u^{(1)}=-u$ in this case. Therefore, the transformation has no sense to a value of $k$ on the continuous spectrum. Thus, our analysis focuses on the case when $k$ is outside the continuous spectrum, i.e., for $k\in \mathbb{C}_I$.

 With Proposition \ref{p1} and $\operatorname{Re}(k_1)>0$, we note that
 \begin{equation}
m_{k_1}(\eta, \eta)=0\Leftrightarrow\eta=0.
 \end{equation}
$\eta^{\prime}(x_0)=0$ follows $\eta(x_0)=0$ at $x_0 \in \mathbb{R}$, which implies $\eta(x)=0$ for every $x \in \mathbb{R}$. As $\varphi_-\left(x ; k_1\right)$ satisfies the nonzero asymptotic limit (\ref{j1j}) as $x \rightarrow-\infty$, then $\eta(x)=\varphi_-\left(x ; k_1\right) e^{-i k_1^2 x} \neq 0$ and $m_{k_1}(\eta, \eta) \neq 0$ for every finite $x \in \mathbb{R}$. It is suffice to replace $m_{k_1}(\eta,\eta)$ by $m_{k_1}(\varphi_-, \varphi_-)$.

According to (\ref{asc}) and (\ref{csc}), if $a(k_1) \neq 0$, there exists $a>0$ such that
\begin{equation}
 \left|m_{k_1}(\varphi_-, \varphi_-)\right| \geq a,\label{mkg}
\end{equation}
for all $x \in \mathbb{R}$.
In fact, as
\begin{equation}
\lim_{x \rightarrow-\infty}m_{k_1}(\varphi_-, \varphi_-)=k_1,\label{mk1j}
\end{equation}
and with (\ref{mkg}), $m_{k_1}(\varphi_-, \varphi_-)$ may only tend to zero when $x \rightarrow+\infty$. However, it follows from the representation (\ref{j1j}) that
\begin{equation}
\lim_{x \rightarrow+\infty}\phi_+(x ; k_1) =e_2,
\end{equation}
and the fact that $\varphi_-(\cdot ; k_1) \in L^{\infty}(\mathbb{R})$ imply that
\begin{equation}
\lim_{x \rightarrow+\infty}\varphi_{-,1}(x ; k_1)=a(k_1),\nonumber
\end{equation}
so that
\begin{equation}
\lim_{x \rightarrow+\infty}m_{k_1}(\varphi_-, \varphi_-)\neq  0.
\end{equation}
Therefore,  (\ref{mkg}) is true. With the triangle inequality, the bounds (\ref{xv1})-(\ref{xv2}) of Proposition \ref{p52}, the bound (\ref{mkg}), and $\left|A_{k_1}(\eta)\right|=1$, we obtain
\begin{equation}
  \begin{aligned}
\left\|u^{(1)}\right\|_{L^{2,1}}&\leq\|u\|_{L^{2,1}}+\left\|C_{k_1}(\varphi_-)\right\|_{L^{2,1}}\leq\|u\|_{L^{2,1}}\\
&+2 a^{-1}\left|k_1^2-\bar{k}_1^2\right|\left\|\varphi_{-,1}(\cdot, k_1) \overline{\varphi_{-,2}\left(\cdot, k_1\right)}\right\|_{L^{2,1}}<\infty.\label{u2k}
  \end{aligned}
\end{equation}
The norms $\left\|\partial_x u^{(1)}\right\|_{L^{2,1}}$, $\left\|\partial_x^2 u^{(1)}\right\|_{L^2}$ and $\left\|\partial_x^3 u^{(1)}\right\|_{L^2}$ are estimated similarly with the bounds (\ref{xv1})-(\ref{xv2}) and (\ref{mkg}).

If $a(k_1)=0$, the uniform bound (\ref{mkg}) is no longer valid since
\begin{equation}
\lim_{x \rightarrow+\infty}m_{k_1}(\varphi_-, \varphi_-)=0.
\end{equation}
The proof on the estimate (\ref{u2k}) is done on the interval $(-\infty, R)$ with arbitrary $R>0$. To extend the estimate (\ref{u2k}) on the interval $(R, \infty)$, we use
\begin{equation}
  \varphi_-(x ; k_j) e^{-i k_j^2 x}=\gamma \phi_+(x ; k_j) e^{i k_j^2 x}, \quad x \in \mathbb{R}
  \end{equation}
and write $\eta(x)=\varphi_-(x ;k_1) e^{-i k_1^2 x}=\gamma \phi_+(x ; k_1) e^{i k_1^2 x}$.

Therefore, $u^{(1)}=D(\varphi_-,k_1) u$ can be rewritten as the second equation in (\ref{u2fj}). As
\begin{equation}
\lim_{x \rightarrow+\infty}m_{k_1}(\phi_+, \phi_+)=\bar{k}_1,
\end{equation}
we use the same method on the interval $(R, \infty)$ by using the equivalent representation of $u^{(1)}$.
\end{proof}

\section{Global   well-posedness   with solitons}
\label{sec4}
\subsection{Time-evolution of  Darboux transformation}
    \hspace*{\parindent}

We first explain the idea for the $1$-soliton and then the statement on the case of fintely many solitons is proved iteratively.
\label{sec41}
    \hspace*{\parindent}

Let $u(t, \cdot) \in \mathcal{Z}_1 \subset H^3(\mathbb{R}) \cap H^{2,1}(\mathbb{R})$ be a local solution of the Cauchy problem (\ref{cs1}) on $(-T, T)$ for some $T>0$. For every fixed time $t \in(-T, T)$, we obtain a new potential $u^{(1)}(t, \cdot)=D(\eta(t, \cdot), k_1) u(t, \cdot)$ of the spectral problem (\ref{laxn1}) by means of the Darboux transformation. If $k_1 \in \mathbb{C}_I$ is taken such that $a(k_1)=0$, then $u^{(1)}(t, \cdot) \in \mathcal{Z}_0$. On the other hand, let $\hat{u}(t, \cdot) \in \mathcal{Z}_0$ be a solution to the Cauchy problem (\ref{cs1}) which satisfies the initial condition $\hat{u}(0, \cdot)=u^{(1)}(0, \cdot) \in \mathcal{Z}_0$, then the solution $\hat{u}(t, \cdot) \in \mathcal{Z}_0$ exists for every $t \in \mathbb{R}$, in particular, for $t\in(-T,T)$ \cite{cqy,laff2}.

Therefore, the core is to prove $\hat{u}(t,\cdot)=u^{(1)}(t,\cdot)$ for every $t \in(-T, T)$, which first requires the following lemma on the identity of the scattering data for the two potentials.
\begin{lemma}
  \label{71}
For every $t \in(-T, T)$, the potentials $\hat{u}(t, \cdot)$ and $u^{(1)}(t, \cdot)$ produce the same scattering data.
\end{lemma}
\begin{proof}
As both potentials $\hat{u}(t, \cdot)$ and $u^{(1)}(t, \cdot)$ remain in $\mathcal{Z}_0$ for every $t \in(-T, T)$, then the scattering data consist only of the reflection coefficient in \cite{laff2}. For the potential $u(t, \cdot) \in \mathcal{Z}_1$ with $t \in(-T, T)$, we have $r(t;k)=b(t;k) / a(t;k)$ for $k \in \mathbb{R} \cup i \mathbb{R}$. By denoting
\begin{equation}
r^{(1)}(t;k)=b^{(1)}(t;k) / a^{(1)}(t;k),
\end{equation}
we have the reflection coefficient of $u^{(1)}(t, \cdot) \in \mathcal{Z}_0$ for $t \in(-T, T)$. Lemma \ref{l41} characterizes how the two reflection coefficients are related with one another:
\begin{equation}
r^{(1)}(t;k)=-r(t;k) \frac{k_1^2}{\bar{k}_1^2} \frac{k^2-\bar{k}_1^2}{k^2-k_1^2}, \quad k \in \mathbb{R} \cup i \mathbb{R}, \quad t \in(-T, T).\label{rb2}
\end{equation}
When $u(x,t)$ is a solution to the FL equation in $\mathcal{Z}_1$, we may derive the time evolution of the reflection coefficient $r(t;k)$ as
\begin{equation}
r(t;k)=r(0;k) e^{2i\eta^2 t}, \quad t \in(-T, T),\label{rr0}
\end{equation}
which, with the aid of (\ref{rb2}), implies that
\begin{equation}
r^{(1)}(t;k)=r^{(1)}(0;k) e^{2 i \eta^2 t}, \quad t \in(-T, T).\label{rr0t}
\end{equation}
Note that the equation (\ref{rr0}) coincides with the case without solitons in \cite{laff2}.

For the reflection coefficient $\hat{r}$ of the potential $\hat{u}$, we know $r^{(1)}(0, k)=\hat{r}(0, k)$ since $u^{(1)}(0, \cdot)=\hat{u}(0, \cdot)$. By using the time evolution of the reflection coefficient from \cite{laff2} and the expression (\ref{rr0t}), we obtain
\begin{equation}
\hat{r}(t;k)=\hat{r}(0;k) e^{2 i \eta^2 t}=r^{(1)}(0;k) e^{2 i \eta^2 t}=r^{(1)}(t;k), \quad t \in(-T, T),
\end{equation}
which proves the lemma.
\end{proof}

\begin{proposition}
  \label{72}
The potential $u^{(1)}(t, \cdot)=D (\eta(t,\cdot),k_1) u(t,\cdot)$ is a new solution of the FL equation for $t \in(-T, T)$.
\end{proposition}
\begin{proof}
In   \cite{laff2}, the existence and the Lipschitz continuity of the mapping $L^{2,1}(\mathbb{R} \cup i \mathbb{R}) \supset U \ni r \mapsto u \in \mathcal{Z}_0 \subset H^3(\mathbb{R}) \cap H^{2,1}(\mathbb{R})$ are established by means of the solvability of the associated Riemann-Hilbert problem. Therefore, the same reflection coefficients is mapped to the same solution and $\hat{u}(t, \cdot)=u^{(1)}(t, \cdot)$ for every $t \in(-T, T)$ with the result in Lemma \ref{71}. As $\hat{u}$ is a solution of the FL equation, so does $u^{(1)}$.
\end{proof}

\subsection{Existence of global solution}
    \hspace*{\parindent}

Combing Lemma \ref{71} and \ref{72}, we immediately obtain the following conclusion.
\begin{proposition}
Fix $k_1 \in \mathbb{C}_I$. Given a local solution $u(t, \cdot) \in H^3(\mathbb{R}) \cap H^{2,1}(\mathbb{R})$, $t \in(-T, T)$ to the Cauchy problem (\ref{cs1}) for some $T>0$, we define
\begin{equation}
\eta(t, x):=\varphi_-(t, x ; k_1) e^{-i(k_1^2 x+\eta^2(k_1)t)},
\end{equation}
where $\varphi_-$is the Jost function of the linear system (\ref{cslp}) and (\ref{cslp2}). Then, $u^{(1)}(t, \cdot)=D (\eta(t, \cdot),k_1) u(t, \cdot)$ belongs to $H^3(\mathbb{R}) \cap H^{2,1}(\mathbb{R})$ for every $t \in[0, T)$ and satisfies the Cauchy problem (\ref{cs1}) for $u^{(1)}(0,\cdot)=D (\eta(0, \cdot),k_1) u(0, \cdot)$.
\label{t1}
\end{proposition}
\begin{proof}
  The proof of Theorem \ref{t1} in the case of finitely many solitons relies on the iterative use of the argument above. For a given $u \in \mathcal{Z}_N, N \in \mathbb{N}$, we remove the distinct eigenvalues $\left\{k_1, \ldots k_N\right\}$ in $\mathbb{C}_I$ by iterating the Darboux transformation $N$ times. We set $u^{(0)}=u$ and
$$
u^{(l)}=D \left(\eta^{(l-1)},k_1\right) u^{(l-1)}, \quad(1 \leq l \leq N),
$$
which eventually constructs $u^{(N)} \in Z_0$. The arguments of Lemma \ref{71} and \ref{72} apply to the last potential $u^{(N)}$. As a result, the $N$-fold iteration of the Darboux transformation of a solution $u(t, \cdot) \in \mathcal{Z}_N$ of the Cauchy problem (\ref{cs})-(\ref{cs1}) for $t \in[0, T)$ produces a new solution $u^{(N)}(t, \cdot) \in \mathcal{Z}_0$ of the Cauchy problem (\ref{cs})-(\ref{cs1}). Figure \ref{fiey} gives the proof diagram.
\end{proof}

\noindent{\bf Proof of the Theorem \ref{74}}

\begin{proof}
  Let $u_0 \in \mathcal{Z}_1 \subset H^3(\mathbb{R}) \cap H^{2,1}(\mathbb{R})$ and $k_1 \in \mathbb{C}_I$ be the only root of $a(k)$ in $\mathbb{C}_I$. By Lemma \ref{l41}, if $\eta(x)=\varphi_-(x ; k_1) e^{-i k_1^2 x}$, where $\varphi_-$ is the Jost function of the spectral problem (\ref{cslp}) associated with $u_0$, then $u_0^{(1)}=D (\eta,k_1) u_0$ belongs to $\mathcal{Z}_0 \subset H^3(\mathbb{R}) \cap H^{2,1}(\mathbb{R})$. Also, the mapping is invertible with $u_0=D (\widetilde{\eta},k_1) u_0^{(1)}$, where $\widetilde{\eta}$ is expressed in term of the new Jost functions $\varphi_-^{(1)}$ and $\phi_+^{(1)}$ by the decomposition formula (\ref{exx1}).

  Let $T>0$ be the maximal existence time for the solution $u(t, \cdot) \in \mathcal{Z}_1, t \in(-T, T)$ to the Cauchy problem (\ref{cs})-(\ref{cs1}) with the initial data $u_0 \in \mathcal{Z}_1$ and the eigenvalue $k_1$. Then, for every fixed $t \in(-T, T)$, the solution $u(t, \cdot) \in \mathcal{Z}_1$ admits the Jost functions $\left\{\varphi_{pm}(t, x ; k), \phi^{\pm}(t, x ; k)\right\}$. For every $t \in(-T, T)$, we define $u^{(1)}$ by the Darboux transformation
\begin{equation}
  u^{(1)}:=D (\eta,k_1) u, \quad \eta(t,x;k):=\varphi_-\left(t,x; k_1\right) e^{-i(k_1^2 x+\eta^2(k_1)t)}
\end{equation}
  where the boundary conditions (\ref{j1j}) are used in the definition of $\varphi_-(t, x ; k_1)$ for every $t \in(-T, T)$.

  By Lemma \ref{72}, $u^{(1)}(t, \cdot) \in \mathcal{Z}_0, t \in(-T, T)$ is a solution of the Cauchy problem (\ref{cs})-(\ref{cs1}) with the initial data $u_0^{(1)} \in \mathcal{Z}_0$. By the existence and uniqueness results \cite{cqy,laff2}, the solution $u^{(1)}(t, \cdot) \in \mathcal{Z}_0$ is uniquely continued for every $t \in \mathbb{R}$. Let $\left\{\varphi_\pm^{(1)}(t, x ; k), \phi_\pm^{(1)}(t, x ; k)\right\}$ be the Jost functions for $u^{(1)}(t, x)$. For every $t \in(-T, T)$, we have $u=D (\eta^{(1)},k_1) u^{(1)}$ with
\begin{equation}
  \widetilde{\eta}(t, x)=\frac{\gamma e^{-i\left(k_1^2 x+\eta^2t\right)}}{\bar{k}_1 a^{(1)}(k_1)} \varphi_-^{(1)}(t, x ; k_1)+\frac{e^{i\left(k_1^2 x+2\eta(k_1)^2t\right)}}{\bar{k}_1 a^{(1)}(k_1)} \phi_+^{(1)}(t, x ; k_1),
\end{equation}
  where $a^{(1)}(k_1) \neq 0$ thanks to Lemma \ref{l41}.

On the other hand, as $u^{(1)}(t, \cdot) \in \mathcal{Z}_0$ exists for every $t \in \mathbb{R}$, the associated Jost functions $\{\varphi_\pm^{(1)}(t, x ;k), \phi_\pm^{(1)}(t, x ; k)\}$ exist for every $t \in \mathbb{R}$. Therefore, we can define $\tilde{u}=D  (\widetilde{\eta},k_1) u^{(1)},\ t \in \mathbb{R}$.
As $u(t, \cdot)=\hat{u}(t, \cdot) \in \mathcal{Z}_1$ for every $t \in(-T, T)$ by uniqueness, the extended function $\tilde{u}$ is unique to the solution $u$ of the same Cauchy problem (\ref{cs})-(\ref{cs1}), which exists globally in time thanks to the bound (\ref{lk1e}) which is proved in Lemma \ref{34}. Indeed, by \cite{laff2} we have $\left\|u^{(1)}(t, \cdot)\right\|_{H^3 \cap H^{2,1}} \leq M_T$ for every $t \in(-T, T)$, where $T>0$ is arbitrary and $M_T$ depends on $T$. Next, with the bound in Lemma \ref{53}, we have $\|u(t, \cdot)\|_{H^3 \cap H^{2,1}} \leq$ $C_{M_T}$ for every $t \in(-T, T)$. Thus, the solution will not blow up in a finite time and hence there exists a unique global solution $u(t, \cdot) \in \mathcal{Z}_1, t \in \mathbb{R}$ to the Cauchy problem (\ref{cs})-(\ref{cs1}) for every $u_0 \in \mathcal{Z}_1 \subset H^3(\mathbb{R}) \cap H^{2,1}(\mathbb{R})$.

By iterating the Darboux transformation $N$ times and by the same argument as above, we prove the global existence of $u(t, \cdot) \in \mathcal{Z}_N \subset H^3(\mathbb{R}) \cap H^{2,1}(\mathbb{R}), t \in \mathbb{R}$ from the global existence of $u^{(N)}(t, \cdot) \in \mathcal{Z}_0 \subset H^3(\mathbb{R}) \cap H^{2,1}(\mathbb{R}), t \in \mathbb{R}$. This completes the proof of Theorem \ref{74}.
\end{proof}

    \noindent\textbf{Acknowledgements}

    This work is supported by  the National Natural Science
    Foundation of China (Grant No. 11671095, 51879045).\vspace{2mm}

    \noindent\textbf{Data Availability Statements}

    The data which supports the findings of this study is available within the article.\vspace{2mm}

    \noindent{\bf Conflict of Interest}

    The authors have no conflicts to disclose.

\end{document}